\documentclass[12pt]{article}
\pdfoutput=1

\usepackage[english]{babel}

\usepackage[a4paper]{geometry}
\usepackage{amssymb,amsmath,amsthm}
\usepackage{graphicx,cite,color}
\usepackage{longtable,pdflscape,booktabs,caption,multicol}

\usepackage[colorlinks=true, citecolor=black, linkcolor=black, urlcolor=blue]{hyperref}
\usepackage{enumitem}
\usepackage{mathtools}

\theoremstyle{plain}
\newtheorem{theorem}{Theorem}
\newtheorem{lemma}[theorem]{Lemma}

\theoremstyle{definition}

\theoremstyle{remark}

\usepackage{float}
\usepackage{tikz}
\usepackage{subfig}
\usetikzlibrary{automata}
\usetikzlibrary{decorations.pathmorphing}
\tikzset{snake it/.style={decorate, decoration=snake}}

\def\dj{d\kern-0.4em\char"16\kern-0.1em}
\def\Dj{\hbox{\raise0.3ex\hbox{-}\kern-0.4em  D}}

\usepackage{authblk}

\title{An inverse result for\\ Wang's theorem on extremal trees}

\author[1, 2]{Ivan Damnjanovi\'c\thanks{The first author is supported by Diffine LLC.}}
\affil[1]{Faculty of Electronic Engineering, University of Ni\v s}
\affil[2]{Diffine LLC}
\author[3]{\v{Z}arko Ran\dj elovi\'c}
\affil[3]{Centre for Mathematical Sciences, University of Cambridge}
\affil[ ]{{\tt ivan@diffine.com}, {\tt zr233@cam.ac.uk}}

\date{}

\newcommand{\doi}[1]{DOI: \href{https://doi.org/#1}{\texttt{#1}}}

\begin{document}

\maketitle

\begin{abstract}
Among all trees on $n$ vertices with a given degree sequence, how do we maximise or minimise the sum over all adjacent pairs of vertices $x$ and $y$ of $f(\deg x, \deg y)$? Here $f$ is a fixed symmetric function satisfying a `monotonicity' condition that
\[
    f(x, a) + f(y, b) > f(y, a) + f(x, b) \quad \mbox{for any $x > y$ and $a > b$} .
\]
These functions arise naturally in several areas of graph theory, particularly chemical graph theory.

Wang showed that the so-called `greedy' tree maximises this quantity, while an `alternating greedy' tree minimises it. Our aim in this paper is to solve the inverse problem: we characterise precisely which trees are extremal for these two problems.

\bigskip\noindent
{\bf Mathematics Subject Classification:} 05C35, 05C05, 05C07, 05C92.\\
{\bf Keywords:} tree, degree sequence, adjacent vertices, graph invariant, algorithm, construction, extremal problem.
\end{abstract}

\pagebreak
\section{Introduction}\label{intro}

Let $T_D$ denote the set of all trees on $n$ vertices that have a fixed degree sequence
$D = ( d_1, d_2, \ldots, d_n )$. For a symmetric function $f \colon \mathbb{N} \times \mathbb{N} \to \mathbb{R}$, and a tree $T$ in $T_D$, we write
 $R_f(T)$ for the sum 
 \begin{equation}\label{rf_def}
    R_f(T) = \sum_{u \sim v} f(\deg_G(u), \deg_G(v)).
\end{equation}
Such graph invariants are very natural in their own right, and are studied considerably in chemical graph theory, where they are typically referred to as topological indices and are applied to describe a particular structural property of a given graph of interest \cite{Gutman1, Randic, Fajtlowicz, Milicevic, Zhou, Gutman2, Gutman3}. See Table \ref{ti_examples} for some examples.

\begin{table}[H]
\begin{center}
{\scriptsize
\begin{tabular}{lll}
\toprule $R_f$ & $f(x, y)$\\
\midrule
Randi\'c index & $\dfrac{1}{\sqrt{xy}}$\\[1.3em]
first Zagreb index & $x + y$\\[0.8em]
second Zagreb index & $xy$\\[0.8em]
second modified Zagreb index & $\dfrac{1}{xy}$\\[1.3em]
geometric--arithmetic index & $\dfrac{2\sqrt{xy}}{x+y}$\\[1.3em]
harmonic index & $\dfrac{2}{x+y}$ &\\[1.3em]
sum--connectivity index & $\dfrac{1}{\sqrt{x+y}}$\\[1.3em]
atom--bond connectivity index & $\sqrt{\dfrac{x+y-2}{xy}}$\\[1.3em]
Sombor index & $\sqrt{x^2+y^2}$\\
\bottomrule
\end{tabular}
}
\caption{Some topological indices $R_f$ together with their corresponding functions $f(x, y)$.}
\label{ti_examples}
\end{center}
\end{table}
Actually, usually these functions $f$ have a monotonicity property that we shall refer to as `positive polarity', which says that
\begin{align}
    \label{cond_c2}f(x, a) + f(y, b) \ge f(y, a) + f(x, b) \quad \mbox{for any $x > y$ and $a > b$}.
\end{align}
We will also say that $f$ satisfies `strict positive polarity' if
\begin{align}
    \label{cond_c1}f(x, a) + f(y, b) > f(y, a) + f(x, b) \quad \mbox{for any $x > y$ and $a > b$}.
\end{align}

\pagebreak\noindent
We mention in passing that positive polarity often arises because $f$ is the restriction to $\mathbb{N} \times \mathbb{N}$ of a function $g \colon [1, +\infty) \times [1, +\infty) \to \mathbb{R}$ such that the mixed second derivative $\dfrac{\partial^2 f}{\partial x \, \partial y}(x, y)$ exists and is positive on $(1,\infty) \times (1,\infty)$. For such a $g$ it is easy to check that the restriction $f$ satisfies polarity, and this perhaps explains why polarity is so frequently present.

In an earlier paper, Wang \cite{Wang} investigated the positive polarity functions and provided two algorithms --- one that constructs a single tree maximizing the $R_f$ value on $\mathcal{T}_D$ and another that constructs one or more trees that minimize the $R_f$ value on $\mathcal{T}_D$ \cite[Theorem 1.1]{Wang}, for any such function $f$. Our research is primarily motivated by these two algorithms and it is our central goal to extend the said results by providing a way to construct the full solution set to the according extremal problems. Bearing this in mind, we offer the following two non-deterministic tree construction algorithms that yield a $\mathcal{T}_D$ tree for a given non-increasing degree sequence $D = (d_1, d_2, \ldots, d_n) \in \mathbb{N}^n$ such that $\mathcal{T}_D \neq \varnothing$.

\paragraph{Algorithm 1}
\begin{enumerate}[label=\textbf{(\roman*)}]
    \item Add a new vertex, assign its desired degree value to $d_1$ and assign its availability value to $d_1$ as well.
    \item For $j = \overline{2, n}$, repeat the following steps until an output tree is reached.
    \begin{enumerate}[label=\textbf{(\arabic*)}]
        \item Add a new vertex $u$ and assign its desired degree and availability values both to $d_j$.
        \item Let $X$ be the set of all the vertices different from $u$ that have a positive availability.
        \item Choose a vertex $v$ from $X$ so that this vertex has the greatest possible desired degree among all the vertices from $X$.
        \item Add an edge whose endpoints are the vertices $u$ and $v$ and decrease the availabilities of these two vertices by one.
    \end{enumerate}
\end{enumerate}

\paragraph{Algorithm 2}
\begin{enumerate}[label=\textbf{(\roman*)}]
    \item For each $j = \overline{1, n}$, add some new vertex, assign its desired degree value to $d_j$ and assign its availability value to $d_j$ as well.
    \item Repeat the following steps until exactly $n-1$ edges have been added so that an output tree is reached.
    \begin{enumerate}[label=\textbf{(\arabic*)}]\label{algo_2_thing}
        \item Let the set $X$ comprise all the pairs $(u, v)$ of vertices with positive availabilities such that:
        \begin{enumerate}[label=\textbf{(\alph*)}]
            \item $u$ has the minimum possible desired degree out of all the vertices that have a positive availability;
            \item $u$ and $v$ do not belong to the same component and the sums of availabilities across the respective components where $u$ and $v$ belong are not both equal to one, unless these are the only two components.
        \end{enumerate}
        \item Choose an element of $X$, i.e.\ some $(u_0, v_0) \in X$, so that $v_0$ has the greatest possible desired degree among all the $v$ vertices in the $(u, v)$ pairs of $X$.
        \item Add an edge whose endpoints are the vertices $u_0$ and $v_0$ and decrease the availabilities of these two vertices by one.
    \end{enumerate}
\end{enumerate}

It is worth pointing out that, provided $\mathcal{T}_D \neq \varnothing$, Algorithm 1 is clearly well defined. Also, after each \ref{algo_2_thing} iteration from Algorithm 2, the total availabilities of all the components must always yield a valid tree degree sequence, again due to $\mathcal{T}_D \neq \varnothing$. For this reason, it is not difficult to see that the corresponding set $X$ can never be empty. This observation assures us that Algorithm 2 is also well defined. We now present the main result of the given paper in the form of the following theorem.
\begin{theorem}\label{main_theorem}
    For some $n \in \mathbb{N}$, let $D \in \mathbb{N}^n$ be a non-increasing sequence of $n$ integers such that $\mathcal{T}_D \neq \varnothing$ and let $f \colon \mathbb{N} \times \mathbb{N} \to \mathbb{R}$ be a discrete symmetric function. We then have:
    \begin{enumerate}[label=\textbf{(\roman*)}]
        \item If $f$ is a strict positive polarity function, then a tree $T \in \mathcal{T}_D$ attains the maximum $R_f$ value on $\mathcal{T}_D$ if and only if it is constructible by Algorithm~1 and it attains the minimum $R_f$ value on $\mathcal{T}_D$ if and only if it is constructible by Algorithm~2.
        \item If $f$ is a positive polarity function, then any tree constructible by Algorithm~1 attains the maximum $R_f$ value on $\mathcal{T}_D$ and any tree constructible by Algorithm~2 attains the minimum $R_f$ value on $\mathcal{T}_D$.
    \end{enumerate}
\end{theorem}

The remainder of the paper will focus on providing a full proof of Theorem~\ref{main_theorem}. Its structure will be organized as follows. Section \ref{preliminaries} will serve to introduce certain preliminary remarks, as well as some auxiliary construction-related terms for the purpose of making the rest of the proof more concise and easier to follow. Afterwards, Sections \ref{sc_algo_1} and \ref{sc_algo_2} will be used to prove the validity of Algorithms 1 and 2, respectively. Finally, Section \ref{conclusion} will finish the paper by disclosing a brief conclusion regarding all the newly obtained results and will give some examples that elaborate how the given algorithms can be used.   

\pagebreak
We use standard notation where for a graph $G$, the order is $|G|$ and $E(G)$ is its set of edges. Also, we will consider all graphs to be undirected, finite and simple. Moreover, we shall implement $\deg_G(u)$ in order to signify the degree of some vertex $u$ from the graph $G$. Finally, it is worth pointing out that all results are trivial for $n = 1$, so we will always assume that $n \ge 2$.

\section{Preliminaries}\label{preliminaries}

First of all, it is not difficult to demonstrate that the second claim stated in Theorem \ref{main_theorem} quickly follows from the first. Let $f$ be an arbitrarily chosen positive polarity function and let $D \in \mathbb{N}^n$ be a non-increasing degree sequence such that $\mathcal{T}_D \neq \varnothing$. For any $x, y, a, b \in \mathbb{N}$ such that $x > y$ and $a > b$, we have
\begin{alignat*}{2}
    && (x - y)(a - b) &> 0\\
    \implies \quad && xa + yb &> ya + xb,
\end{alignat*}
which means that for any parameter $t \in \mathbb{R},\, t > 0$, the discrete symmetric function $f_t(x, y) = f(x, y) + txy$ is surely a strict positive polarity function. According to the first statement from Theorem \ref{main_theorem}, we have that any tree $T_0 \in \mathcal{T}_D$ constructible by Algorithm~1 certainly maximizes the $R_{f_t}$ value on $\mathcal{T}_D$, for each $t > 0$. In other words, we get
\begin{equation}\label{aux_1}
    R_{f_t}(T_0) \ge R_{f_t}(T)
\end{equation}
for any $T \in \mathcal{T}_D$ and $t > 0$. Since both sides of Eq.\ (\ref{aux_1}) can be viewed as linear functions in $t$, we are able to simply plug in $t \to 0^+$ in order to reach
\[
    R_f(T_0) \ge R_f(T),
\]
as desired. An analogous argument can be made regarding the $R_f$ minimizing property of any tree constructible by Algorithm 2. Bearing everything in mind, it becomes evident that in order to complete the proof of Theorem \ref{main_theorem}, it is sufficient to prove just the first disclosed statement. For this reason, we shall deal exclusively with strict positive polarity functions $f$ in the remainder of the paper.

Algorithms 1 and 2 represent two tree construction mechanisms that both involve the simple addition of vertices and edges in some particular order. Throughout both algorithms, each vertex is assigned two property values: the desired degree, which signifies the degree that the vertex should have once the construction is completed, and the availability, which determines how many more edges should be incident to the given vertex in order for its degree to match its desired degree, as needed. We will now define certain construction-related auxiliary terms which we will rely on for the sake of making the proof of Theorem \ref{main_theorem} easier to follow. 

We shall refer to the ordered pair $((v_0, v_1, v_2, \ldots, v_{n-1}), (f_1, f_2, \ldots, f_{n-1}))$ as a \emph{scheme} of some tree $T$ of order $n \in \mathbb{N}$ provided that this tree can be obtained via the following simple construction algorithm:
\begin{enumerate}[label=\textbf{(\arabic*)}]
    \item Add the vertex $v_0$.
    \item For each integer $j = \overline{1, n-1}$, add the vertex $v_j$ and an edge whose endpoints are $v_j$ and the previously added vertex $f_j$.
\end{enumerate}

Now, we will use the term \emph{positive availability vertex}, or \emph{PA vertex} for short, to denote a vertex whose availability is greater than zero. If some PA vertex has the greatest desired degree among all the PA vertices, we will then refer to this vertex as a \emph{strong positive availability vertex}, or \emph{SPA vertex} for short. Similarly, if a PA vertex has the smallest desired degree among all the PA vertices, we will then call this vertex a \emph{weak positive availability vertex}, or \emph{WPA vertex} for short.

For a given component, we will use the term \emph{total availability} to refer to the sum of availabilities of all of its vertices and we shall denote the total availability of some component $C$ by $t(C)$. We will consider a \emph{uniform component} to be a component such that all of its PA vertices have the same desired degree. Moreover, we will use $\deg(C)$ to signify the desired degree of any PA vertex from the uniform component $C$. Furthermore, a uniform component that has the total availability equal to one must necessarily have a single PA vertex, and we will call such a component a \emph{cleaf}. If a component is not uniform, but contains only SPA and WPA vertices, we will then refer to it as a \emph{minimum--maximum mixed component}, or \emph{MMM component} for short. Finally, a component that is neither uniform nor an MMM component shall be called a \emph{forbidden component}.

In the rest of the paper, we will take $D = (d_1, d_2, \ldots, d_n) \in \mathbb{N}^n$ to be an arbitrarily chosen fixed non-increasing sequence of $n$ integers such that $\mathcal{T}_D \neq \varnothing$. Bearing in the mind all the newly introduced terms, it is possible to reformulate Algorithms 1 and 2 in a more concise manner, as demonstrated below.

\paragraph{Algorithm 1}
\begin{enumerate}[label=\textbf{(\roman*)}]
    \item Add a new vertex and assign its desired degree and availability values to $d_1$.
    \item For $j = \overline{2, n}$, repeat the following steps until an output tree is reached.
    \begin{enumerate}[label=\textbf{(\arabic*)}]
        \item Add a new vertex $u$ and assign its desired degree and availability values both to $d_j$.
        \item For an arbitrarily chosen SPA vertex $v \neq u$, add an edge whose endpoints are the vertices $u$ and $v$ and decrease the availabilities of these two vertices by one.
    \end{enumerate}
\end{enumerate}

\paragraph{Algorithm 2}
\begin{enumerate}[label=\textbf{(\roman*)}]
    \item For each $j = \overline{1, n}$, add some new vertex and assign its desired degree and availability values to $d_j$.
    \item Repeat the following steps until exactly $n-1$ edges have been added so that an output tree is reached.
    \begin{enumerate}[label=\textbf{(\arabic*)}]
        \item\label{set_x1} Let the set $X$ comprise all the pairs $(u, v)$ of PA vertices from distinct components such that $u$ is a WPA vertex and the total availabilities of the respective components where $u$ and $v$ belong are not both equal to one, unless these are the only two components.
        \item Choose an element of $X$, i.e.\ some $(u_0, v_0) \in X$, so that $v_0$ has the greatest possible desired degree among all the $v$ vertices in the $(u, v)$ pairs of $X$.
        \item Add an edge whose endpoints are the vertices $u_0$ and $v_0$ and decrease the availabilities of these two vertices by one.
    \end{enumerate}
\end{enumerate}

\section{Validity of Algorithm 1}\label{sc_algo_1}

In this section, we will consider an arbitrary strict positive polarity function $f$ and prove that each tree maximizing $R_f$ on $\mathcal{T}_D$ must be constructible by Algorithm~1. Afterwards, we will swiftly demonstrate the converse as well --- that each tree constructible by Algorithm 1 surely attains the maximimum $R_f$ value on $\mathcal{T}_D$.

To begin, we point out that each tree surely has at least one scheme (see, for example, \cite[Corollary 1.5.2]{diestel}). However, it becomes convenient to notice that the trees that attain the maximum $R_f$ value on $\mathcal{T}_D$ always possess very specific schemes. Our immediate goal shall be to elaborate on this fact and provide a result that will later be used while proving the extremal property of Algorithm 1. We start with the following auxiliary lemma regarding the degrees of vertices that lie on an arbitrary path.

\begin{lemma}\label{path_lemma}
    Let $T \in \mathcal{T}_D$ be a tree that attains the maximum $R_f$ value on $\mathcal{T}_D$ and let $u$ and $v$ be two of its arbitrarily chosen vertices. For any vertex $w$ that lies on the path from $u$ to $v$, we necessarily have
    \[
        \deg_T(w) \ge \min(\deg_T(u), \deg_T(v)) .
    \]
\end{lemma}
\begin{proof}
    We shall prove the lemma by contradiction. Let $P$ be the $(u, v)$-path in $T$ and suppose that there does lie a vertex on $P$ whose degree is below $\min(\deg_T(u), \linebreak \deg_T(v))$. It is straightforward to see that $\deg_T(u), \deg_T(v) \ge 2$ must hold. For this reason, we can construct a non-trivial path $Q$ from $v$ to some leaf $t$ so that this path is entirely disjoint with $P$, except for the vertex $v$.
    
    Now, let $p_1$ be the first vertex on $P$ whose degree is lower than $\min(\deg_T(u), \linebreak \deg_T(v))$, and let $p_0$ be the vertex on this path before it. Similarly, let $q_1$ be the first vertex on $Q$ whose degree is below $\min(\deg_T(u), \deg_T(v))$, and let $q_0$ be the vertex on this path before it. Taking everything into consideration, we obtain a $(p_0, q_1)$-path as depicted in Figure \ref{path_lemma_fig} that will be of further interest.

\begin{figure}[ht]
    \centering
    \begin{tikzpicture}
        \node[state, minimum size=0.75cm, thick] (1) at (-1.0, 0) {$u$};
        \node[state, minimum size=0.75cm, thick] (2) at (1.5, 0) {$p_0$};
        \node[state, minimum size=0.75cm, thick] (3) at (3.0, 0) {$p_1$};
        \node[state, minimum size=0.75cm, thick] (4) at (5.5, 0) {$v$};
        \node[state, minimum size=0.75cm, thick] (5) at (8.0, 0) {$q_0$};
        \node[state, minimum size=0.75cm, thick] (6) at (9.5, 0) {$q_1$};
        \node[state, minimum size=0.75cm, thick] (7) at (12.0, 0) {$t$};

        \path[draw=black, thick, snake it] (1) -- (2);
        \path[thick] (2) edge (3);
        \path[draw=black, thick, snake it] (3) -- (4);
        \path[draw=black, thick, snake it] (4) -- (5);
        \path[thick] (5) edge (6);
        \path[draw=black, thick, snake it] (6) -- (7);
    \end{tikzpicture}
    \caption{The obtained $(p_0, q_1)$-path in $T$, alongside the vertices $u$ and $t$.}
    \label{path_lemma_fig}
\end{figure}
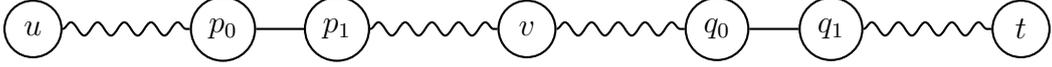

    It is clear that the tree $T$ satisfies
    \begin{align*}
        p_0 \sim p_1, && q_0 \sim q_1, && p_0 \not\sim q_0, && p_1 \not\sim q_1 .
    \end{align*}
    Bearing this in mind, we can remove the edges $p_0 p_1$ and $q_0 q_1$ from $T$ and add the edges $p_0 q_0$ and $p_1 q_1$ in order to obtain another tree $T_1$ whose vertices have the same degrees as in $T$. Hence, $T_1 \in \mathcal{T}_D$. Furthermore, $R_f(T)$ and $R_f(T_1)$ will have the same summands in Eq.\ (\ref{rf_def}) except for those that correspond to the deleted and newly added edges. This immediately implies
    \begin{align}\label{aux_3}
        \begin{split}
        R_f(T_1) - R_f(T) &= f(\deg_T(p_0), \deg_T(q_0)) + f(\deg_T(p_1), \deg_T(q_1))\\
        &\quad - f(\deg_T(p_0), \deg_T(p_1)) - f(\deg_T(q_0), \deg_T(q_1)) .
        \end{split}
    \end{align}
    However, we know that
    \begin{align*}
        \deg_T(p_0), \deg_T(q_0) &\ge \min(\deg_T(u), \deg_T(v)),\\
        \deg_T(p_1), \deg_T(q_1) &< \min(\deg_T(u), \deg_T(v)),
    \end{align*}
    which swifty leads us to
    \begin{align*}
        f(\deg_T(p_0), \deg_T(q_0)) &+ f(\deg_T(q_1), \deg_T(p_1)) >\\
        &>  f(\deg_T(q_1), \deg_T(q_0)) + f(\deg_T(p_0), \deg_T(p_1))
    \end{align*}
    by virtue of Eq.\ (\ref{cond_c1}). Now, Eq.\ (\ref{aux_3}) tells us that $R_f(T_1) - R_f(T) > 0$ must hold, which is impossible since the tree $T$ attains the maximum $R_f$ value on $\mathcal{T}_D$. Hence, we obtain a contradiction.
\end{proof}

Now, by taking into consideration Lemma \ref{path_lemma}, we are able to formulate and prove the next lemma regarding the constructibility of trees that attain the maximum $R_f$ value.

\begin{lemma}\label{constructibility_lemma}
    If $T \in \mathcal{T}_D$ is some tree that attains the maximum $R_f$ value on $\mathcal{T}_D$, then this tree surely has a scheme $((v_0, v_1, v_2, \ldots, v_{n-1}), (f_1, f_2, \ldots, f_{n-1}))$ such that
    \begin{itemize}
        \item for each $j = \overline{0, n-1}$, we have $\deg_T(v_j) = d_{j+1}$;
        \item for all the $1 \le j < h \le n-1$ such that $\deg_T(v_j) = \deg_T(v_h)$, the condition $\deg_T(f_j) \ge \deg_T(f_h)$ must hold.
    \end{itemize}
\end{lemma}
\begin{proof}
    Lemma~\ref{path_lemma} tells us that, for each $j = \overline{\min D, \max D}$, the subgraph of $T$ induced by the set of vertices whose degree is at least $j$ must be a tree. From here, we quickly conclude that we can construct $T$ by simply constructing its subtree induced by the vertices of degree $\max D$, then extending this subtree to the subtree induced by the vertices of degree at least $\max D - 1$, and so on, until we obtain $T$ itself. Thus, the tree $T$ necessarily possesses a scheme $C' = ((v'_0, v'_1, v'_2, \ldots, v'_{n-1}), (f'_1, f'_2, \ldots, f'_{n-1}))$ such that the degrees of the vertices $v'_0, v'_1, v'_2, \ldots, v'_{n-1}$ appear in non-increasing order. This promptly implies $\deg_T(v'_j) = d_{j+1}$ for each $j = \overline{0, n-1}$.

    We have obtained a scheme $C'$ that satisfies the first condition given in the lemma. In order to finalize the proof, we will explain how this scheme can be modified so that the second condition surely holds as well. First of all, it is easy to check that the second condition necessarily holds for the vertices of degree $\max D$, hence it becomes sufficient to show that, for any $\beta,\, \min D \le \beta < \max D$, the addition of vertices of degree $\beta$ within the scheme $C'$ can be permuted in some manner so that the second condition becomes satisfied.

    The key observation to make is that while $T$ is constructed via the algorithm dictated by $C'$, each vertex of degree $\beta$ is surely connected to a vertex of degree at least $\beta$ upon being added. Moreover, each vertex of degree $\beta$ that is connected to a vertex of degree greater than $\beta$ can certainly freely be reordered among all the vertices of degree $\beta$. In other words, this vertex can be added before or after any other vertex of degree $\beta$, given the fact that its initial neighbor is definitely present to begin with. This directly means that we can reorder the addition of all the vertices of degree $\beta$ so that we first add those whose initial neighbor has the greatest possible degree, then those whose initial neighbor has the second greatest degree, and so on, until we add the vertices of degree $\beta$ whose initial neighbor also has the degree $\beta$, and which cannot freely be reordered. By applying the said transformation on $C'$ for each possible $\beta,\, \min D \le \beta < \max D$, we obtain a scheme $C$ that truly satisfies both criteria given in the lemma, which completes the proof.
\end{proof}

By implementing Lemma \ref{constructibility_lemma}, we can immediately prove one half of the desired extremal property of Algorithm 1. This result is disclosed within the following lemma.

\begin{lemma}\label{main_lemma}
    Any tree that attains the maximum $R_f$ value on $\mathcal{T}_D$ must be constructible by Algorithm 1.
\end{lemma}
\begin{proof}
    Let $T$ be any such tree. It is clear that this tree must have a scheme $C$ that satisfies the criteria stated in Lemma \ref{constructibility_lemma}. Now, while $T$ is being constructed via the algorithm dictated by $C$, suppose that there exists a vertex $v$ such that, upon being added, it is not adjacent to a pre-existing SPA vertex. Let $p$ be such a pre-existing vertex and let $q$ be the vertex that $v$ gets connected to instead. Due to the criteria imposed on $C$ by virtue of Lemma \ref{constructibility_lemma}, we see that none of the vertices of degree $\deg_T(v)$ that are added after $v$ can be adjacent to $p$ either, which means that the vertex $p$ necessarily has a neighbor $u$ in $T$ such that $\deg_T(u) < \deg_T(v)$. Taking everything into consideration, we obtain that the tree $T$ bears a structure as demonstrated in Figure \ref{main_lemma_fig}.

\begin{figure}[ht]
    \centering
    \begin{tikzpicture}
        \node[state, minimum size=0.75cm, thick] (1) at (0.0, 0) {$u$};
        \node[state, minimum size=0.75cm, thick] (2) at (1.5, 0) {$p$};
        \node[state, minimum size=0.75cm, thick] (3) at (4.0, 0) {$q$};
        \node[state, minimum size=0.75cm, thick] (4) at (5.5, 0) {$v$};

        \path[thick] (1) edge (2);
        \path[draw=black, thick, snake it] (2) -- (3);
        \path[thick] (3) edge (4);
    \end{tikzpicture}
    \caption{The structure of the tree $T$.}
    \label{main_lemma_fig}
\end{figure}
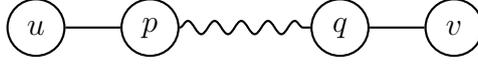

    It is obvious that the tree $T$ satisfies
    \begin{align*}
        u \sim p, && v \sim q, && u \not\sim q, && v \not\sim p .
    \end{align*}
    If we remove the edges $up$ and $vq$ from $T$ and add the edges $uq$ and $vp$, we get another tree $T_1$ whose vertices have the same degrees as in $T$. For this reason, we have $T_1 \in \mathcal{T}_D$. Using the same logic as in the proof of Lemma \ref{path_lemma}, it is easy to show that
    \begin{align}\label{aux_2}
        \begin{split}
        R_f(T_1) - R_f(T) &= f(\deg_T(u), \deg_T(q)) + f(\deg_T(v), \deg_T(p))\\
        &\quad - f(\deg_T(u), \deg_T(p)) - f(\deg_T(v), \deg_T(q)) .
        \end{split}
    \end{align}
    Taking into consideration that
    \[
        \deg_T(p) > \deg_T(q) \ge \deg_T(v) > \deg_T(u) ,
    \]
    it becomes straightforward to obtain
    \begin{align*}
        f(\deg_T(p), \deg_T(v)) &+ f(\deg_T(q), \deg_T(u)) >\\
        &>  f(\deg_T(q), \deg_T(v)) + f(\deg_T(p), \deg_T(u))
    \end{align*}
    by directly implementing Eq.\ (\ref{cond_c1}). Now, by using Eq.\ (\ref{aux_2}), this immediately leads us to $R_f(T_1) - R_f(T) > 0$, which is clearly not possible due to the fact that $T$ attains the maximum $R_f$ value on $\mathcal{T}_D$.
    
    Thus, we conclude that while $T$ is being constructed in accordance with the scheme $C$, the vertices must be added in such a way their degrees yield a non-increasing sequence, with each vertex after the first being connected to a pre-existing SPA vertex. However, this is precisely how Algorithm 1 works, hence it promptly follows that $T$ must indeed be constructible by Algorithm 1.
\end{proof}

We are now finally in position to put all the pieces of the puzzle together and complete the proof of the validity of Algorithm 1.

\bigskip\noindent
\emph{Proof of the validity of Algorithm 1}.\quad
If a tree attains the maximum $R_f$ value on $\mathcal{T}_D$, then it is surely constructible by Algorithm 1, by virtue of Lemma \ref{main_lemma}. Thus, in order to finish the validity proof, we need to show that each tree constructible by Algorithm 1 must also attain the maximum $R_f$ value on $\mathcal{T}_D$. Since there are finitely many isomorphism classes among the $\mathcal{T}_D$ trees, there certainly exists a tree $T_0 \in \mathcal{T}_D$ that attains the maximum $R_f$ value on $\mathcal{T}_D$. Due to Lemma \ref{main_lemma}, we know that $T_0$ is constructible by Algorithm 1. From here we notice that in order to demonstrate that all the trees constructible by Algorithm 1 attain the maximum $R_f$ value on $\mathcal{T}_D$, it is sufficient to prove that they all have the same $R_f$ value.

For each $1 \le j \le n$ and $0 \le k \le n-1$, let $Y_{j, k}$ denote the sum of availabilities of all the existing vertices of degree $k$ after $j$ vertices have been added in total while executing Algorithm 1. Let the scheme $((v_0, v_1, v_2, \ldots, v_{n-1}), (f_1, f_2, \ldots, f_{n-1}))$ correspond to an execution of Algorithm 1 which yields the tree $T \in \mathcal{T}_D$. It becomes apparent that while adding vertex $v_j$, the degrees of $v_j$ and $f_j$ can be determined by using the simple expression
\begin{align*}
    \deg_T v_j &= d_j,\\
    \deg_T f_j &= \max \{ k \in \mathbb{N} \colon 0 \le k \le n-1,\, Y_{j, k} > 0 \} . 
\end{align*}
Besides that, it is possible to obtain the values $Y_{j + 1, k}$ in terms of $Y_{j, k}$ by simply setting $Y_{j + 1, k} \coloneqq Y_{j, k}$ for each $0 \le k \le n-1$, then increasing the value of $Y_{j + 1, \deg_T v_j}$ by $\deg_T v_j - 1$ and then decreasing the value of $Y_{j + 1, \deg_T f_j}$ by one. Here, it is important to notice that regardless of how the algorithm is executed, the elements $Y_{j, k}$ depend solely on the given degree sequence $D$, and not the concrete execution itself. For this reason, the degrees of $v_0, v_1, v_2, \ldots, v_{n-1}$ and $f_1, f_2, \ldots, f_{n-1}$ must be the same in all the executions of Algorithm 1. Given the fact that for any $T \in \mathcal{T}_D$ constructed via the scheme $((v_0, v_1, v_2, \ldots, v_{n-1}), (f_1, f_2, \ldots, f_{n-1}))$, we have
\[
    R_f(T) = \sum_{j = 1}^{n-1} f(\deg_T(v_j), \deg_T(f_j)),
\]
it is clear that all the trees constructible by Algorithm 1 must attain the same $R_f$ value, as desired. \hfill\qed

\section{Validity of Algorithm 2}\label{sc_algo_2}

In this section, we will consider an arbitrary strict positive polarity function $f$ and prove that each tree minimizing $R_f$ on $\mathcal{T}_D$ must be constructible by Algorithm~2. We will then show that each tree constructible by Algorithm 2 also attains the minimum value of $R_f$ on $\mathcal{T}_D$, thereby completing the proof. We begin by disclosing the following two auxiliary lemmas.

\begin{lemma}\label{path_swap_lemma}
    Let $T \in \mathcal{T}_D$ be a tree that attains the minimum $R_f$ value on $\mathcal{T}_D$. If the tree $T$ contains a path $x_0 x_1 \cdots x_{n-1} x_n$ of length $n \in \mathbb{N}, \, n \ge 3$ which satisfies $\deg_T(x_0) < \deg_T(x_n)$, then $\deg_T(x_1) \ge \deg_T(x_{n-1})$ must be true.
\end{lemma}
\begin{proof}
\begin{figure}[ht]
    \centering
    \begin{tikzpicture}
        \node[state, minimum size=1.10cm, thick] (1) at (0.0, 0) {$x_0$};
        \node[state, minimum size=1.10cm, thick] (2) at (1.8, 0) {$x_1$};
        \node[state, minimum size=1.10cm, thick] (3) at (4.8, 0) {$x_{n-1}$};
        \node[state, minimum size=1.10cm, thick] (4) at (6.6, 0) {$x_n$};

        \path[thick] (1) edge (2);
        \path[draw=black, thick, snake it] (2) -- (3);
        \path[thick] (3) edge (4);
    \end{tikzpicture}
    \caption{The structure of the path $P = x_0 x_1 \cdots x_{n-1} x_n$.}
    \label{good_lemma_fig}
\end{figure}
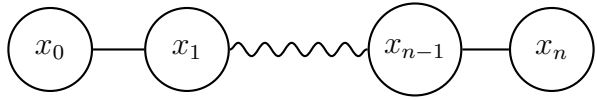
    Assume the contrary and let $P = x_0 x_1 \cdots x_{n-1} x_n$ be a path in $T$ of length $n \in \mathbb{N}, \, n \ge 3$  such that $\deg_T(x_0) < \deg_T(x_n)$ and $\deg_T(x_1) < \deg_T(x_{n-1})$. Now consider the graph $T_1$ obtained by removing the edges $x_0 x_1, x_{n-1} x_n$ and adding the edges $x_0 x_{n-1}$ and $x_1 x_n$. Bearing in mind Figure \ref{good_lemma_fig}, it is evident that $T_1$ must be a tree. Moreoever, it is straightforward to see that $T_1 \in \mathcal{T}_D$. By implementing Eq.\ (\ref{rf_def}), we immediately obtain that
    \begin{align*}
        R_f(T) - R_f(T_1) &= f(\deg_T(x_n) ,\deg_T(x_{n-1})) + f(\deg_T(x_0), \deg_T(x_1))\\
        & \qquad - f(\deg_T(x_n), \deg_T(x_1)) - f(\deg_T(x_0), \deg_T(x_{n-1})) .
    \end{align*}
    Now, it is sufficient to use Eq.\ (\ref{cond_c1}) in order to reach $R_f(T) - R_f(T_1) > 0$. Thus, $T$ does not attain the minimum value of $R_f$ on $\mathcal{T}_D$, which is a contradiction.
\end{proof}

\begin{lemma}\label{good_lemma}
Let $T\in \mathcal{T}_D$ be a tree that attains the minimum $R_f$ value on $\mathcal{T}_D$. Now, let $n, b, a$ be positive integers and suppose that $u, v_0,v_1$ are vertices in $T$ such that $\deg_T(u)=a, \, \deg_T(v_0)=b, \deg_T(v_1)=a$. Furthermore, let $c=\min(a,b)$ and $d=\max(a,b)$. For an arbitrary edge $z_1z_2\in E(T)$, say that it is good if $\{\deg_T(z_1),\deg_T(z_2)\}=\{c,d\}$. If there is an $i \in \{0,1\}$ such that there is a path $ux_1 \cdots x_nv_i \cdots v_{1-i}y$ in $T$ with $\deg_T(x_n),\deg_T(y)\in [c,d]$, then one of the edges $x_nv_i, v_{1-i}y$ must be good.
\end{lemma}
\begin{proof}
    The proof is trivial to do if $a = b$. We now choose to carry out the proof only for the case when $a < b$, given the fact that the statement can be proved in an entirely analogous manner whenever $b < a$. Thus, we will assume that $c=a$ and $b=d$.

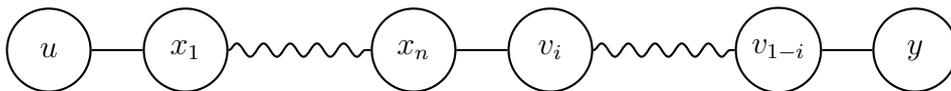
\begin{figure}[ht]
    \centering
    \begin{tikzpicture}
        \node[state, minimum size=1.10cm, thick] (1) at (0.0, 0) {$u$};
        \node[state, minimum size=1.10cm, thick] (2) at (1.8, 0) {$x_1$};
        \node[state, minimum size=1.10cm, thick] (3) at (4.8, 0) {$x_n$};
        \node[state, minimum size=1.10cm, thick] (4) at (6.6, 0) {$v_i$};
        \node[state, minimum size=1.10cm, thick] (5) at (9.6, 0) {$v_{1-i}$};
        \node[state, minimum size=1.10cm, thick] (6) at (11.4, 0) {$y$};

        \path[thick] (1) edge (2);
        \path[draw=black, thick, snake it] (2) -- (3);
        \path[thick] (3) edge (4);
        \path[draw=black, thick, snake it] (4) -- (5);
        \path[thick] (5) edge (6);
    \end{tikzpicture}
    \caption{The structure of the path $P = ux_1 \cdots x_nv_i \cdots v_{1-i}y$.}
    \label{good_lemma_fig_2}
\end{figure}

Suppose the contrary, that neither $x_nv_i$ nor $v_{1-i}y$ are good edges. Define the paths $P,P'$ to be $P=ux_1...x_nv_i...v_{1-i}y$ and $P'=uv_0...v_1x_1...x_ny$ where we have $V(P)=V(P')$. Let $T_1$ be the graph obtained by taking $T$ and replacing the edges $ux_1, x_nv_i, v_{1-i}y$ with the edges $uv_0, v_1x_1, x_ny$. Notice that $T_1$ is obtained from $T$ when we replace $P$ with $P'$. This means that $T_1$ is indeed a tree, and given the fact that $P,P'$ have the same endpoints we have not changed any of the vertex degrees. Hence we obtain that $T_1\in \mathcal{T}_D$. If $i=0$, we then have $a < \deg_T(x_n)$ and $\deg_T(y) < b$, which leads us to
    \begin{align*}
    R_f(T) &- R_f(T_1) = f(a,\deg_T(x_1)) + f(\deg_T(x_n),b) + f(a,\deg_T(y))\\
    &- f(a,b) -f(a,\deg_T(x_1))-f(\deg_T(x_n),\deg_T(y))>0   
    \end{align*}
    by implementing Eq.\ (\ref{cond_c1}) together with the aforementioned inequalities. This is a contradiction since $T$ obtains the minimum value of $R_f$ on $\mathcal{T}_D$. The case when $i = 1$ can be resolved in an analogous manner and we choose to leave out the according proof details.
\end{proof}

In the remainder of the section, we will use $a, b$ to denote the desired degrees of the WPA and SPA vertices, respectively. Our next step shall be to use Lemmas~\ref{path_swap_lemma} and \ref{good_lemma} in order to demonstrate that every $T\in \mathcal{T}_D$ which minimizes $R_f$ is constructible using Algorithm 2. The said result is given in the next lemma.

\begin{lemma}\label{algo_2_construction_lemma}
 Suppose that $T \in \mathcal{T}_D$ minimizes $R_f$ on $\mathcal{T}_D$. Then it is possible to construct $T$ using Algorithm 2 in such a way that at any time after all the vertices have been added and before the tree is fully constructed, the following conditions hold:
\begin{enumerate}[label=\textbf{(\roman*)}]
    \item\label{zh1} There is at most one MMM component and all the other components are uniform.
    \item\label{zh2} If there is an MMM component, then all the cleaves $C$ have $\deg(C)$ equal to $a$ or $b$.
    \item\label{zh3} If there are no MMM components, then each cleaf $C$ such that $\deg(C) > a$ certainly has a degree which is not below the degree of any non-cleaf.
\end{enumerate}
\end{lemma}
\begin{proof}
Suppose that $k\ge 0$ is the maximum number of edges that we can add using Algorithm 2 such that at every required step the conditions \ref{zh1}, \ref{zh2} and \ref{zh3} hold. Note that at the start we have that all components are uniform and all cleaves have the degree one, so \ref{zh1}, \ref{zh2} and \ref{zh3} do hold.  Let $n=|T|$. If $k=n-1$, then we are done. Suppose that $k=n-2$. Then we would have one edge left and we could add it according to Algorithm 2 and there would be a single component remaining so we would be done. Now suppose that $k<n-2$. 

Let $T'\le T$ be the spanning subgraph of $T$ that we can construct using Algorithm 2 with $k=|E(T')|$. Suppose that $T'$ has $l$ components $C_1, C_2, \ldots, C_l$. Consider the graph $T_1$ with vertices $C_i$ and where there is an edge $C_iC_j$ if and only if there is an edge in $T$ between $C_i$ and $C_j$. Note that $T_1$ is connected and acyclic and hence a tree. Thus, we must have $\deg_{T_1}(C_i) = 1$ for some $C_i$. Since there are no $e,f$ such that there are two edges in $T$ between $C_e$ and $C_f$ (as that would give a cycle in $T$), we have that $C_i$ is a cleaf. In particular, there must be a cleaf.

We now choose to split the given problem into two cases.

\bigskip\noindent
\emph{Case 1}.\quad There is no MMM component in $T'$. We split this case into two further subcases.

\medskip\noindent
\emph{Case 1a}.\quad
Every component of degree $a$ or $b$ is a cleaf. Note that by adding any edge from $T \backslash T'$ we do not obtain any component with zero total availability since we could still add edges to make $T$. Thus, there is some component that is not a cleaf. Let $C$ be a component that is not a cleaf and has the highest possible degree and let $\deg(C) = \xi$. Also, let $u$ be a PA vertex of desired degree $a$. There is a shortest path in $T$ from $u$ to $C$. Let that path be $P = u x_1 \cdots x_m v$ with $m\ge 0$ and $v\in C$. Since $C$ is not a cleaf, there must be some $v_0\in C$ and a PA vertex $y\in V(T)\backslash C$ such that $y\not \in P$ and $v_0y\in E(T)$. Therefore, $ux_1 \cdots x_mv \cdots v_0y$ is a path in $T$.

If $m=0$, then we may add the edge $uv$. This is valid in accordance with Algorithm 2 and it is not difficult to realize that all the components would now be uniform. Moreover, the newly formed component containing $u, v$ may or may not be a cleaf, but either way, each cleaf with a degree greater than $a$ would not have a degree lower than any non-cleaf. For this reason, the additional conditions \ref{zh1}, \ref{zh2} and \ref{zh3} would all hold as well. This observation would contradict the maximality of $k$, as desired.

Now, if $m>0$, then note that $\deg_T(x_1)\le \xi$ because $x_1$ must belong to a component that is not a cleaf and not $C$, since $P$ must enter and then leave that component. This means that $\deg_T(u) = a \le \deg_T(y)$, as well as $\deg_T(x_1) \le \xi = \deg_T(v_0)$. This allows us to apply Lemma \ref{path_swap_lemma} on the path $ux_1 \cdots x_mv \cdots v_0y$ and deduce that $\deg_T(y) = a$ or $\deg_T(x_1) = \xi$. In any case, there exists an edge whose endpoints are two PA vertices with the desired degrees $a$ and $\xi$ coming from different components, one of which is not a cleaf. Thus, we can add that edge according to Algorithm 2. It is not difficult to establish that all the newly obtained components will be uniform. Also, the newly formed component may or may not be a cleaf, but either way, the condition \ref{zh3} will hold, as desired. This contradicts the maximality of $k$ once again.

\medskip\noindent
\emph{Case 1b}.\quad
Not all components with degrees $a, b$ are cleaves. As noted earlier, there must exist at least one cleaf. The condition \ref{zh3} guarantees that there certainly exists a cleaf of degree $a$ or $b$. Without loss of generality, let there be a cleaf of degree $a$. We now have that either there is a component of degree $b$ that is not a cleaf, or all the components of degree $b$ are cleaves and then there must be a component of degree $a$ which is not a cleaf. Either way, there are two components $C, D$ of degrees $a, b$, respectively, such that one of them is a cleaf, while the other is not.

Without loss of generality, let $C$ be a cleaf and let $u\in C$ be the corresponding PA vertex. Following the same argument as in Case 1a, we can show that that Algorithm~2 permits us to add an edge of $T$ whose endpoints have the desired degrees $a, b$ in $T$. We now have two possibilities --- either the newly formed component is uniform or not. If it is uniform, then it must be of degree $a$ or $b$ and it is not difficult to check that all the conditions \ref{zh1}, \ref{zh2} and \ref{zh3} must hold. If it is not uniform, this means that we had a non-cleaf of degree $b$ to begin with, which promptly implies that we end up with an MMM component and that all the newly existing cleaves must have the degree $a$ or $b$. This means that the conditions \ref{zh1}, \ref{zh2} and \ref{zh3} all hold. We reach a contradiction regarding the maximality of $k$.

\bigskip\noindent
\emph{Case 2}.\quad
There is an MMM component in $T'$. Let $u$ be any vertex which is in some cleaf, and by \ref{zh2}, without loss of generality, let $\deg_T(u) = a$. Let $C$ be the MMM component and let $u x_1 \cdots x_m v_0$ be the shortest path in $T$ from $u$ to $C$, where $m \ge 0$ and $v_0$ is a PA vertex. Also, let $v_1\in C$ be a PA vertex such that $\deg_T(v_0) \neq \deg_T(v_1)$. There must be some $y\not \in C$ such that $v_1y\in E(T)$, which means that $u x_1 \cdots x_m v_0 \cdots v_1 y$ is a path in $T$.

If $m=0$, then either $\deg_T(v_0) = b$, in which case we can add the edge $u v_0$, or $\deg_T(v_1)=b$, and then by Lemma \ref{path_swap_lemma} we obtain $\deg_T(v_0) = b$ or $\deg_T(y) = a$, thus we can add either the edge $uv_0$ or $v_1y$. In each of these scenarios, Algorithm 2 permits us to add an edge in such a way that all the conditions \ref{zh1}, \ref{zh2} and \ref{zh3} are satisfies. This can be noticed by using a similar argumentation as done so in Case 1b.

If $m\ge 1$, then by Lemma \ref{good_lemma}, at least one of the two edges $x_m v_0$ or $v_1 y$ will have endpoints with desired degrees $a$ and $b$ and could be added using Algorithm~2. Whatever the case, by adding the said edge, we will connect some component to $C$ and keep at most one MMM component. If there is a new cleaf, it must have degree $a$ or $b$. Thus, the conditions \ref{zh1}, \ref{zh2} and \ref{zh3} will all certainly hold. This once again contradicts the maximality of $k$.
\end{proof}

We have just shown that all trees that minimize $R_{f}$ can be constructed using Algorithm 2. Thus, the only thing left to do is to show the converse --- that any tree constructed using Algorithm 2 actually minimizes $R_{f}$. In order to finalize the desired proof, we will rely on the following lemma which analyzes the behavior of Algorithm 2 while it is being executed.

\begin{lemma}\label{behavior_lemma}
    Whenever Algorithm 2 is applied on some degree sequence $D$, at any time after all the vertices have been added and before the tree is fully constructed, the following conditions must hold:
    \begin{enumerate}[label=\textbf{(\roman*)}]
        \item\label{kkk1} There are no forbidden components.
        \item\label{kkk2} If there is an MMM component, then all the cleaves $C$ have $\deg(C)$ equal to $a$ or $b$.
        \item\label{kkk3} If there are no MMM components, then each cleaf $C$ such that $\deg(C) > a$ certainly has a degree which is not below the degree of any non-cleaf.
    \end{enumerate}
\end{lemma}
\begin{proof}
We will prove the lemma by induction. Clearly, the lemma statement is true before any edges have been added. Suppose it is true after adding $k$ edges for some $k < n-2$. By using a similar argument as done so in the proof of Lemma \ref{algo_2_construction_lemma}, there must be a cleaf.

If there is an MMM component, then there must be a cleaf of degree $a$ or $b$, so Algorithm 2 dictates that an edge should be added whose endpoints have desired degrees $a$ and $b$. Regardless of which such edge is added, we obtain that all the cleaves must have degrees $a$ and $b$. Thus, the conditions \ref{kkk1}, \ref{kkk2} and \ref{kkk3} must all be satisfied, as desired.

If there is no MMM component, then if all the uniform components of degree $a$ are cleaves, the problem is straightforward to resolve --- Algorithm 2 dictates that we should add an edge such that its endpoints have the desired degrees $a$ and $\xi$, where $\xi$ represents the greatest degree that a non-cleaf component has. The newly formed component will be uniform and its degree shall be $\xi$, and from here, it is not difficult to notice that \ref{kkk1}, \ref{kkk2} and \ref{kkk3} must all hold.

Finally, if there is no MMM component and there exists at least one component of degree $a$ that is not a cleaf, then Algorithm 2 states that we should add an edge whose endpoints have the desired degrees $a$ and $b$. However, by doing so, we obtain that the new component is either a uniform component of degree $a$ or $b$, or an MMM component. If the component is uniform, then it is trivial to see that all the conditions \ref{kkk1}, \ref{kkk2} and \ref{kkk3} are satisfied. If the newly generated component is an MMM component, this means that it arose from merging a non-cleaf of degree $a$ and a non-cleaf of degree $b$. From here, it is evident that all the remaining cleaves must have the degree $a$ or $b$, which implies that the conditions \ref{kkk1}, \ref{kkk2} and \ref{kkk3} all hold once again.
\end{proof}

We are now in the position to implement Lemma \ref{behavior_lemma} in order to complete the proof of the validity of Algorithm 2, and thereby finalize the proof of Theorem \ref{main_theorem}.

\bigskip\noindent
\emph{Proof of the validity of Algorithm 2}.\quad
If a tree attains the minimum $R_f$ value on $\mathcal{T}_D$, then it must be constructible by Algorithm 2, according to Lemma \ref{algo_2_construction_lemma}. Thus, in order to complete the proof, it is sufficient to demonstrate that each tree constructible by Algorithm 2 must also attain the minimum $R_f$ value on $\mathcal{T}_D$. Similarly to the proof of the validity of Algorithm 1, we note that there are finitely many trees up to isomorphism with a given degree sequence. For this reason, the minimum of $R_f$ must be achieved by some tree and this tree must be constructible by Algorithm 2, due to Lemma \ref{algo_2_construction_lemma}. So, in order to prove that Algorithm 2 only produces trees that minimize $R_f$, we just need to show that any two trees constructed by Algorithm 2 have the same $R_f$.

Let $T, S$ be constructed by Algorithm 2 and let $x_1 y_1, \ldots, x_{n-1} y_{n-1}$ be the edges of $T$ and $z_1 t_1, \ldots, z_{n-1} t_{n-1}$ be the edges of $S$ in the order in which they were added in Algorithm 2 to construct $T, S$, respectively. It is enough to show that, for each $k = \overline{1, n-1}$, we have $\{\deg_T(x_k), \deg_T(y_k)\} = \{\deg_S(z_k), \deg_S(t_k)\}$, since this would clearly indicate $R_f(T) = R_f(S)$, by virtue of Eq.\ (\ref{rf_def}). Now, for any tree $H$ constructed by Algorithm 2, we will denote $Y_{j,k}(H)$ to be the total sum of availabilities of all vertices of desired degree $j$ after $k$ edges have been added. In order to finalize the proof, it becomes sufficient to show that, for a fixed value $k \in \overline{0, n-1}$, $Y_{j, k}(H)$ is the same for all the trees $H$ and all the values of $j$. We shall prove this by induction.

Note that the base case for $k=0$ is true. Suppose that the statement is true up to some $k$, then notice that $Y_{j,k}(S) = Y_{j,k}(T)$. If $a = b$, then it immediately follows that all the edges to be added throughout the rest of the algorithm will necessarily have endpoints whose desired degrees will all be equal to $a$. For this reason, it is clear that $Y_{j,k+1}(S) = Y_{j,k+1}(T)$ will hold for each $j$, and there is nothing left to discuss. Now, suppose that $a < b$. We call an edge \textit{spanning} if its endpoints have degrees $a$ and $b$.  By virtue of Lemma \ref{behavior_lemma}, there is always a cleaf of degree $a$ or $b$ and, thus, the algorithm allows us to only add spanning edges unless there is no MMM component and all the uniform components of degrees $a$ and $b$ are cleaves.

Now consider what happens after adding the first $k$ edges in both $T$ and $S$, for some $k < n - 1$. If we suppose that $Y_{j,k+1}(S) = Y_{j,k+1}(T)$ does not hold for each $j$, we may assume without loss of generality that in $T$, there is no MMM component and all the uniform components of degrees $a$ or $b$ in $T$ are cleaves. We now point out that for any $a < g < b$, the number of uniform components of degree $g$ must be the same in both $T$ and $S$ as at the start of the algorithm after the vertices have been added but the edges have not. To verify this, we observe that Lemma \ref{behavior_lemma} dictates that such uniform components necessarily stay uniform up until $k$ edges have been added. Moreover, the only way for such a component to disappear is if it represents a cleaf which is then merged into another uniform component whose degree is not greater than $a$. However, this is not possible, since in this scenario, the other component would necessarily not be a cleaf, hence it could be merged with one of the components which have a PA vertex whose desired degree is at least $b$. Thus, the said edge addition would not be in accordance with Algorithm~2, which is not possible.

Thus, after $k$ edges have been added, both $S$ and $T$ need to have the same number of uniform components whose degree is $g$, where $a < g < b$. Besides that, the total number of components must be $n-k$ in both trees. Thus, the total number of components containing only PA vertices of desired degrees $a$ and $b$ is the same for $T$ and $S$. It is not difficult to notice that this can only happen if all components containing PA vertices with desired degrees $a$ and $b$ are cleaves in both $T$ and $S$. Now, we can see that neither $x_{k+1} y_{k+1}$ nor $z_{k+1} t_{k+1}$ are spanning. Without loss of generality, let $\deg_T(x_{k+1}) = \deg_S(z_{k+1}) = a$. We further have that $\deg_T(y_k)$ is equal to the smallest $j$, $a \le j < b$ such that $Y_{j,k}(T)$ is strictly larger than the number of uniform components of degree $j$. However, this value is the same for $S$ and $T$, hence $\deg_T(y_k) = \deg_S(t_k)$, which completes the proof. \hfill\qed

\section{Conclusion}\label{conclusion}

Theorem \ref{main_theorem} offers a complete solution set for both the $R_f$ maximization and $R_f$ minimization problem on $\mathcal{T}_D$ whenever the discrete symmetric function $f$ is a strict positive polarity or strict negative polarity function. This result makes a substantial contribution to the field of chemical graph theory due to the sheer fact that many adjacent vertex degree based topological indices are yielded by such functions $f$. For example, by analyzing the topological indices displayed in Table~\ref{ti_examples}, it is straightforward to deduce that a tree $T \in \mathcal{T}_D$ is constructible by Algorithm 1 (Algorithm 2) if and only if it maximizes (minimizes) the Randi\'c index, second Zagreb index, second modified Zagreb index, harmonic index and sum--connectivity index, and if and only if it minimizes (maximizes) the atom--bond connectivity index and Sombor index.

Furthermore, Theorem \ref{main_theorem} offers a partial solution to the $R_f$ maximization and $R_f$ minimization problem on $\mathcal{T}_D$ in the scenario when $f$ is a non-strict positive polarity or negative polarity function. The corresponding results once again represent an improvement over the theorem obtained by Wang \cite[Theorem 1.1]{Wang}. In fact, it is not that convenient to make further analysis for non-strict positive polarity or negative polarity functions since it is very difficult to deduce whether a tree not constructible by Algorithm 1 or 2 is extremal. For example, the first Zagreb index
\[
    M_1(G) = \sum_{u \sim v} \left( \deg_G(u) + \deg_G(v) \right)
\]
can be alternatively written as
\[
    M_1(G) = \sum_{u} \deg_G(u)^2,
\]
where the summing is done over all the graph vertices $u$. From here, it immediately follows that all the trees from $\mathcal{T}_D$ have the same first Zagreb index, hence each tree must be an extremal tree. For this reason, attempting to determine the full solution set for the corresponding extremal problems is very challenging without making additional assumptions regarding the behavior of $f$. The same can be said for the case when $f$ is neither a positive polarity nor a negative polarity function, such as the function $f$ yielding the geometric--arithmetic index, as depicted in Table \ref{ti_examples}.

We finish the paper by giving a brief example of how Theorem \ref{main_theorem} can be used on a concrete valid degree sequence in order to yield the complete solution for the $R_f$ maximization and $R_f$ minimization problem in the case that $f$ is a strict positive polarity function. Let $D_1 = (4, 4, 3, 3, 2, 1, 1, 1, 1, 1, 1, 1, 1)$. By implementing the theorem derived by Wang \cite[Theorem 1.1]{Wang}, we conclude that the greedy tree depicted in Figure \ref{greedy_wang} surely attains the maximum $R_f$ value on $\mathcal{T}_{D_1}$.

However, if we apply Theorem \ref{main_theorem}, we are able to obtain a much stronger result. More precisely, we conclude that some tree $T$ attains the maximum $R_f$ value on $\mathcal{T}_{D_1}$ if and only if it belongs to one of the three isomorphism classes shown in Figure \ref{main_example}. This observation is straightforward to notice --- the trees constructible by Algorithm 1 are precisely such that the two vertices of degree four are adjacent, while each vertex of degree three or two must have a neighbor of degree four. Thus, there essentially exist exactly three different trees that attain the maximum $R_f$ value on $\mathcal{T}_{D_1}$, with the greedy tree obtained by Wang corresponding to the tree given in Figure \ref{first_sol}.

Now, let $D_2 = (8, 7, 6, 6, 5, 5, 3, 3, 3, 2, \underbrace{1, 1, \ldots, 1}_{\mbox{30 ones}})$. If we run Algorithm 2 and use Theorem \ref{main_theorem}, it is possible to obtain nine different isomorphism classes which represent the complete solution set to the $R_f$ minimization problem on $\mathcal{T}_{D_2}$, provided $f$ is a strict positive polarity function. All of these isomorphism classes are depicted in Figure \ref{algo_2_example}. On the other hand, the alternating greedy tree construction given by Wang yields only two isomorphism classes \cite[Figure 7]{Wang} that correspond to the trees given in Figures \ref{algo_2_example_e} and \ref{algo_2_example_h}. From here, it becomes clear that Theorem \ref{main_theorem} represents a substantial improvement over the aforementioned earlier construction mechanism.

\begin{figure}[H]
    \centering
    \begin{tikzpicture}
        \node[state, minimum size=1cm, thick] (1) at (0.0, 0) {$6$};
        \node[state, minimum size=1cm, thick] (2) at (1.5, 0) {$7$};
        \node[state, minimum size=1cm, thick] (3) at (3.0, 0) {$8$};
        \node[state, minimum size=1cm, thick] (4) at (4.5, 0) {$9$};
        \node[state, minimum size=1cm, thick] (5) at (6.0, 0) {$10$};
        \node[state, minimum size=1cm, thick] (6) at (7.5, 0) {$11$};
        \node[state, minimum size=1cm, thick] (7) at (9.0, 0) {$12$};
        \node[state, minimum size=1cm, thick] (8) at (10.5, 0) {$13$};
        
        \node[state, minimum size=1cm, thick] (9) at (1.5, 2) {$2$};
        \node[state, minimum size=1cm, thick] (10) at (5.25, 2) {$3$};
        \node[state, minimum size=1cm, thick] (11) at (8.25, 2) {$4$};
        \node[state, minimum size=1cm, thick] (12) at (10.5, 2) {$5$};

        \node[state, minimum size=1cm, thick] (13) at (6.75, 4) {$1$};

        \path[thick] (1) edge (9);
        \path[thick] (2) edge (9);
        \path[thick] (3) edge (9);
        \path[thick] (4) edge (10);
        \path[thick] (5) edge (10);
        \path[thick] (6) edge (11);
        \path[thick] (7) edge (11);
        \path[thick] (8) edge (12);
        \path[thick] (9) edge (13);
        \path[thick] (10) edge (13);
        \path[thick] (11) edge (13);
        \path[thick] (12) edge (13);
    \end{tikzpicture}
    \caption{The greedy tree for $D_1 = (4, 4, 3, 3, 2, 1, 1, 1, 1, 1, 1, 1, 1)$.}
    \label{greedy_wang}
\end{figure}
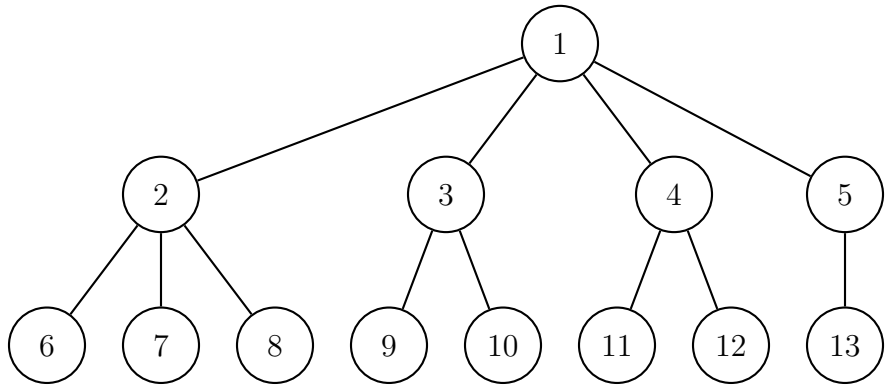

\begin{figure}[H]
    \centering
    \subfloat[] {
        \begin{tikzpicture}
            \node[state, minimum size=0.25cm, thick, fill=black] (1) at (0.0, 0) {$ $};
            \node[state, minimum size=0.25cm, thick, fill=black] (2) at (1.2, 0) {$ $};
            \node[state, minimum size=0.25cm, thick, fill=red] (3) at (0, 1.2) {$ $};
            \node[state, minimum size=0.25cm, thick, fill=red] (4) at (0, -1.2) {$ $};
            \node[state, minimum size=0.25cm, thick, fill=yellow] (5) at (-1.2, 0) {$ $};
            
            \node[state, minimum size=0.25cm, thick] (6) at (2.4, 0) {$ $};
            \node[state, minimum size=0.25cm, thick] (7) at (1.2, 1.2) {$ $};
            \node[state, minimum size=0.25cm, thick] (8) at (1.2, -1.2) {$ $};
            
            \node[state, minimum size=0.25cm, thick] (9) at (0, 2.4) {$ $};
            \node[state, minimum size=0.25cm, thick] (10) at (-1.2, 1.2) {$ $};
            \node[state, minimum size=0.25cm, thick] (11) at (0, -2.4) {$ $};
            \node[state, minimum size=0.25cm, thick] (12) at (-1.2, -1.2) {$ $};
    
            \node[state, minimum size=0.25cm, thick] (13) at (-2.4, 0) {$ $};
    
            \path[thick] (1) edge (2);
            \path[thick] (1) edge (3);
            \path[thick] (1) edge (4);
            \path[thick] (1) edge (5);
            \path[thick] (2) edge (6);
            \path[thick] (2) edge (7);
            \path[thick] (2) edge (8);
            \path[thick] (3) edge (9);
            \path[thick] (3) edge (10);
            \path[thick] (4) edge (11);
            \path[thick] (4) edge (12);
            \path[thick] (5) edge (13);
        \end{tikzpicture}
        \label{first_sol}
    }
    \hspace{1.0cm}
    \subfloat[] {
        \begin{tikzpicture}
            \node[state, minimum size=0.25cm, thick, fill=black] (1) at (0.0, 0) {$ $};
            \node[state, minimum size=0.25cm, thick, fill=black] (2) at (1.2, 0) {$ $};
            \node[state, minimum size=0.25cm, thick, fill=red] (3) at (0, 1.2) {$ $};
            \node[state, minimum size=0.25cm, thick, fill=red] (4) at (0, -1.2) {$ $};
            \node[state, minimum size=0.25cm, thick] (5) at (-1.2, 0) {$ $};
            
            \node[state, minimum size=0.25cm, thick, fill=yellow] (6) at (2.4, 0) {$ $};
            \node[state, minimum size=0.25cm, thick] (7) at (1.2, 1.2) {$ $};
            \node[state, minimum size=0.25cm, thick] (8) at (1.2, -1.2) {$ $};
            
            \node[state, minimum size=0.25cm, thick] (9) at (0, 2.4) {$ $};
            \node[state, minimum size=0.25cm, thick] (10) at (-1.2, 1.2) {$ $};
            \node[state, minimum size=0.25cm, thick] (11) at (0, -2.4) {$ $};
            \node[state, minimum size=0.25cm, thick] (12) at (-1.2, -1.2) {$ $};
    
            \node[state, minimum size=0.25cm, thick] (13) at (3.6, 0) {$ $};
    
            \path[thick] (1) edge (2);
            \path[thick] (1) edge (3);
            \path[thick] (1) edge (4);
            \path[thick] (1) edge (5);
            \path[thick] (2) edge (6);
            \path[thick] (2) edge (7);
            \path[thick] (2) edge (8);
            \path[thick] (3) edge (9);
            \path[thick] (3) edge (10);
            \path[thick] (4) edge (11);
            \path[thick] (4) edge (12);
            \path[thick] (6) edge (13);
        \end{tikzpicture}
    }
    \\
    \subfloat[] {
        \begin{tikzpicture}
            \node[state, minimum size=0.25cm, thick, fill=black] (1) at (0.0, 0) {$ $};
            \node[state, minimum size=0.25cm, thick, fill=black] (2) at (1.2, 0) {$ $};
            \node[state, minimum size=0.25cm, thick, fill=red] (3) at (0, 1.2) {$ $};
            \node[state, minimum size=0.25cm, thick] (4) at (0, -1.2) {$ $};
            \node[state, minimum size=0.25cm, thick, fill=yellow] (5) at (-1.2, 0) {$ $};
            
            \node[state, minimum size=0.25cm, thick] (6) at (2.4, 0) {$ $};
            \node[state, minimum size=0.25cm, thick, fill=red] (7) at (1.2, 1.2) {$ $};
            \node[state, minimum size=0.25cm, thick] (8) at (1.2, -1.2) {$ $};
            
            \node[state, minimum size=0.25cm, thick] (9) at (0, 2.4) {$ $};
            \node[state, minimum size=0.25cm, thick] (10) at (-1.2, 1.2) {$ $};
            \node[state, minimum size=0.25cm, thick] (11) at (1.2, 2.4) {$ $};
            \node[state, minimum size=0.25cm, thick] (12) at (2.4, 1.2) {$ $};
    
            \node[state, minimum size=0.25cm, thick] (13) at (-1.2, -1.2) {$ $};
    
            \path[thick] (1) edge (2);
            \path[thick] (1) edge (3);
            \path[thick] (1) edge (4);
            \path[thick] (1) edge (5);
            \path[thick] (2) edge (6);
            \path[thick] (2) edge (7);
            \path[thick] (2) edge (8);
            \path[thick] (3) edge (9);
            \path[thick] (3) edge (10);
            \path[thick] (7) edge (11);
            \path[thick] (7) edge (12);
            \path[thick] (5) edge (13);
        \end{tikzpicture}
    }
    \caption{All three isomorphism classes corresponding to the trees constructible by Algorithm 1 for $D_1 = (4, 4, 3, 3, 2, 1, 1, 1, 1, 1, 1, 1, 1)$. The vertices of degree four, three and two are colored in black, red and yellow, respectively.}
    \label{main_example}
\end{figure}
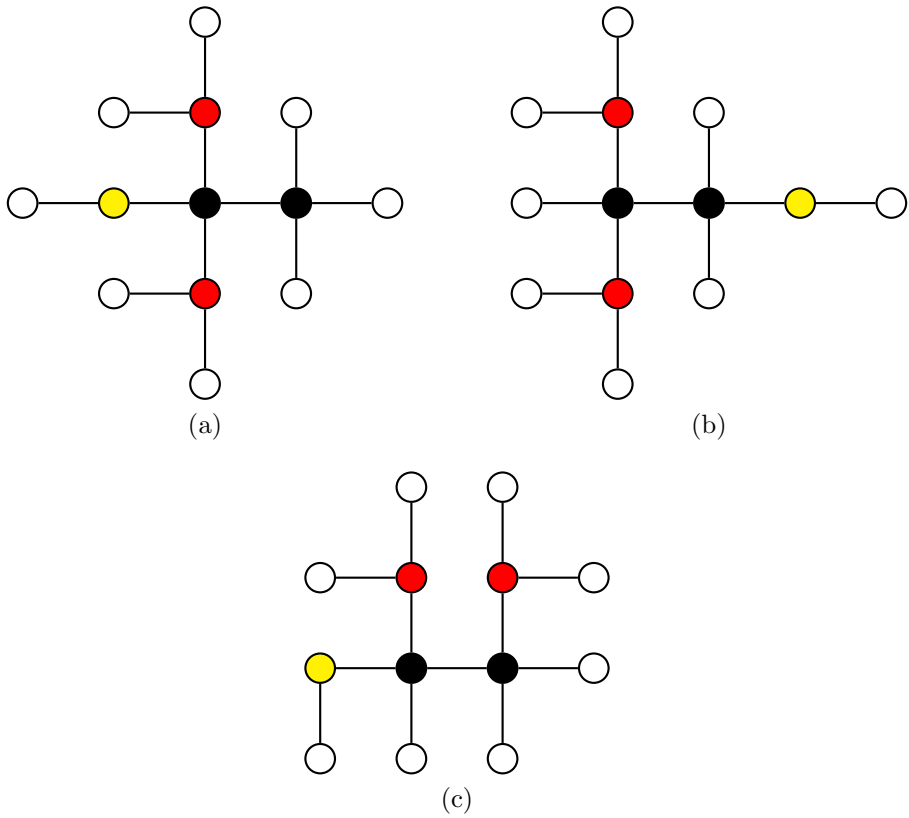

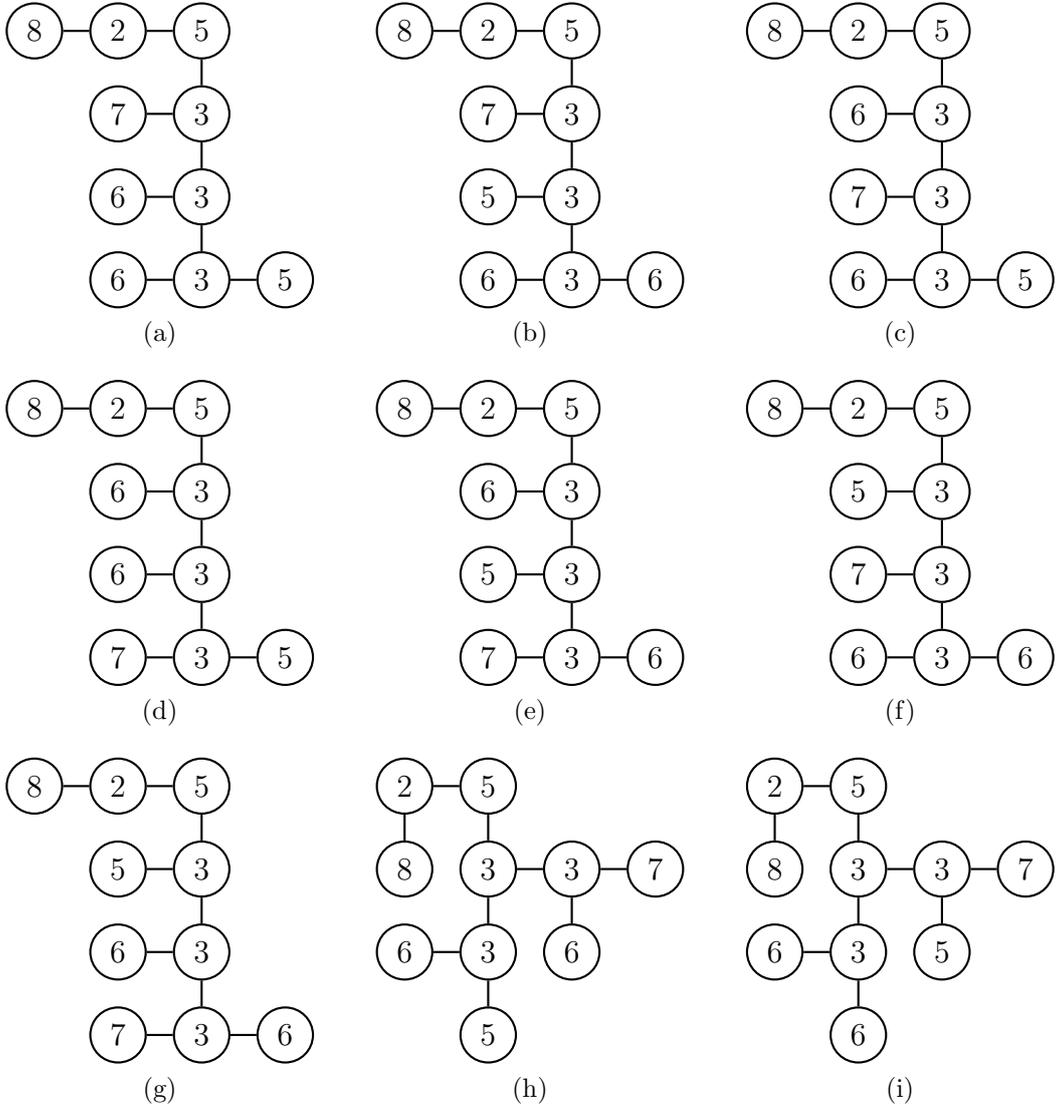
\begin{figure}
    \centering
    \subfloat[] {
        \begin{tikzpicture}
            \node[state, minimum size=0.25cm, thick] (1) at (0, 0) {$8$};
            \node[state, minimum size=0.25cm, thick] (2) at (1.1, 0) {$2$};
            \node[state, minimum size=0.25cm, thick] (3) at (2.2, 0) {$5$};
            \node[state, minimum size=0.25cm, thick] (4) at (2.2, -1.1) {$3$};
            \node[state, minimum size=0.25cm, thick] (5) at (2.2, -2.2) {$3$};
            \node[state, minimum size=0.25cm, thick] (6) at (2.2, -3.3) {$3$};

            \node[state, minimum size=0.25cm, thick] (7) at (1.1, -1.1) {$7$};
            \node[state, minimum size=0.25cm, thick] (8) at (1.1, -2.2) {$6$};
            \node[state, minimum size=0.25cm, thick] (9) at (1.1, -3.3) {$6$};
            \node[state, minimum size=0.25cm, thick] (10) at (3.3, -3.3) {$5$};
            
            \path[thick] (1) edge (2);
            \path[thick] (2) edge (3);
            \path[thick] (3) edge (4);
            \path[thick] (4) edge (5);
            \path[thick] (5) edge (6);

            \path[thick] (4) edge (7);
            \path[thick] (5) edge (8);
            \path[thick] (6) edge (9);
            \path[thick] (6) edge (10);
        \end{tikzpicture}
    }
    \hspace{0.4cm}
    \subfloat[] {
        \begin{tikzpicture}
            \node[state, minimum size=0.25cm, thick] (1) at (0, 0) {$8$};
            \node[state, minimum size=0.25cm, thick] (2) at (1.1, 0) {$2$};
            \node[state, minimum size=0.25cm, thick] (3) at (2.2, 0) {$5$};
            \node[state, minimum size=0.25cm, thick] (4) at (2.2, -1.1) {$3$};
            \node[state, minimum size=0.25cm, thick] (5) at (2.2, -2.2) {$3$};
            \node[state, minimum size=0.25cm, thick] (6) at (2.2, -3.3) {$3$};

            \node[state, minimum size=0.25cm, thick] (7) at (1.1, -1.1) {$7$};
            \node[state, minimum size=0.25cm, thick] (8) at (1.1, -2.2) {$5$};
            \node[state, minimum size=0.25cm, thick] (9) at (1.1, -3.3) {$6$};
            \node[state, minimum size=0.25cm, thick] (10) at (3.3, -3.3)  {$6$};
            
            \path[thick] (1) edge (2);
            \path[thick] (2) edge (3);
            \path[thick] (3) edge (4);
            \path[thick] (4) edge (5);
            \path[thick] (5) edge (6);

            \path[thick] (4) edge (7);
            \path[thick] (5) edge (8);
            \path[thick] (6) edge (9);
            \path[thick] (6) edge (10);
        \end{tikzpicture}
    }
    \hspace{0.4cm}
    \subfloat[] {
        \begin{tikzpicture}
            \node[state, minimum size=0.25cm, thick] (1) at (0, 0) {$8$};
            \node[state, minimum size=0.25cm, thick] (2) at (1.1, 0) {$2$};
            \node[state, minimum size=0.25cm, thick] (3) at (2.2, 0) {$5$};
            \node[state, minimum size=0.25cm, thick] (4) at (2.2, -1.1) {$3$};
            \node[state, minimum size=0.25cm, thick] (5) at (2.2, -2.2) {$3$};
            \node[state, minimum size=0.25cm, thick] (6) at (2.2, -3.3) {$3$};

            \node[state, minimum size=0.25cm, thick] (7) at (1.1, -1.1) {$6$};
            \node[state, minimum size=0.25cm, thick] (8) at (1.1, -2.2) {$7$};
            \node[state, minimum size=0.25cm, thick] (9) at (1.1, -3.3) {$6$};
            \node[state, minimum size=0.25cm, thick] (10) at (3.3, -3.3)  {$5$};
            
            \path[thick] (1) edge (2);
            \path[thick] (2) edge (3);
            \path[thick] (3) edge (4);
            \path[thick] (4) edge (5);
            \path[thick] (5) edge (6);

            \path[thick] (4) edge (7);
            \path[thick] (5) edge (8);
            \path[thick] (6) edge (9);
            \path[thick] (6) edge (10);
        \end{tikzpicture}
    }\\
    \subfloat[] {
        \begin{tikzpicture}
            \node[state, minimum size=0.25cm, thick] (1) at (0, 0) {$8$};
            \node[state, minimum size=0.25cm, thick] (2) at (1.1, 0) {$2$};
            \node[state, minimum size=0.25cm, thick] (3) at (2.2, 0) {$5$};
            \node[state, minimum size=0.25cm, thick] (4) at (2.2, -1.1) {$3$};
            \node[state, minimum size=0.25cm, thick] (5) at (2.2, -2.2) {$3$};
            \node[state, minimum size=0.25cm, thick] (6) at (2.2, -3.3) {$3$};

            \node[state, minimum size=0.25cm, thick] (7) at (1.1, -1.1) {$6$};
            \node[state, minimum size=0.25cm, thick] (8) at (1.1, -2.2) {$6$};
            \node[state, minimum size=0.25cm, thick] (9) at (1.1, -3.3) {$7$};
            \node[state, minimum size=0.25cm, thick] (10) at (3.3, -3.3) {$5$};
            
            \path[thick] (1) edge (2);
            \path[thick] (2) edge (3);
            \path[thick] (3) edge (4);
            \path[thick] (4) edge (5);
            \path[thick] (5) edge (6);

            \path[thick] (4) edge (7);
            \path[thick] (5) edge (8);
            \path[thick] (6) edge (9);
            \path[thick] (6) edge (10);
        \end{tikzpicture}
    }
    \hspace{0.4cm}
    \subfloat[] {
        \begin{tikzpicture}
            \node[state, minimum size=0.25cm, thick] (1) at (0, 0) {$8$};
            \node[state, minimum size=0.25cm, thick] (2) at (1.1, 0) {$2$};
            \node[state, minimum size=0.25cm, thick] (3) at (2.2, 0) {$5$};
            \node[state, minimum size=0.25cm, thick] (4) at (2.2, -1.1) {$3$};
            \node[state, minimum size=0.25cm, thick] (5) at (2.2, -2.2) {$3$};
            \node[state, minimum size=0.25cm, thick] (6) at (2.2, -3.3) {$3$};

            \node[state, minimum size=0.25cm, thick] (7) at (1.1, -1.1) {$6$};
            \node[state, minimum size=0.25cm, thick] (8) at (1.1, -2.2) {$5$};
            \node[state, minimum size=0.25cm, thick] (9) at (1.1, -3.3) {$7$};
            \node[state, minimum size=0.25cm, thick] (10) at (3.3, -3.3) {$6$};
            
            \path[thick] (1) edge (2);
            \path[thick] (2) edge (3);
            \path[thick] (3) edge (4);
            \path[thick] (4) edge (5);
            \path[thick] (5) edge (6);

            \path[thick] (4) edge (7);
            \path[thick] (5) edge (8);
            \path[thick] (6) edge (9);
            \path[thick] (6) edge (10);
        \end{tikzpicture}
        \label{algo_2_example_e}
    }
    \hspace{0.4cm}
    \subfloat[] {
        \begin{tikzpicture}
            \node[state, minimum size=0.25cm, thick] (1) at (0, 0) {$8$};
            \node[state, minimum size=0.25cm, thick] (2) at (1.1, 0) {$2$};
            \node[state, minimum size=0.25cm, thick] (3) at (2.2, 0) {$5$};
            \node[state, minimum size=0.25cm, thick] (4) at (2.2, -1.1) {$3$};
            \node[state, minimum size=0.25cm, thick] (5) at (2.2, -2.2) {$3$};
            \node[state, minimum size=0.25cm, thick] (6) at (2.2, -3.3) {$3$};

            \node[state, minimum size=0.25cm, thick] (7) at (1.1, -1.1) {$5$};
            \node[state, minimum size=0.25cm, thick] (8) at (1.1, -2.2) {$7$};
            \node[state, minimum size=0.25cm, thick] (9) at (1.1, -3.3) {$6$};
            \node[state, minimum size=0.25cm, thick] (10) at (3.3, -3.3) {$6$};
            
            \path[thick] (1) edge (2);
            \path[thick] (2) edge (3);
            \path[thick] (3) edge (4);
            \path[thick] (4) edge (5);
            \path[thick] (5) edge (6);

            \path[thick] (4) edge (7);
            \path[thick] (5) edge (8);
            \path[thick] (6) edge (9);
            \path[thick] (6) edge (10);
        \end{tikzpicture}
    }\\
    \subfloat[] {
        \begin{tikzpicture}
            \node[state, minimum size=0.25cm, thick] (1) at (0, 0) {$8$};
            \node[state, minimum size=0.25cm, thick] (2) at (1.1, 0) {$2$};
            \node[state, minimum size=0.25cm, thick] (3) at (2.2, 0) {$5$};
            \node[state, minimum size=0.25cm, thick] (4) at (2.2, -1.1) {$3$};
            \node[state, minimum size=0.25cm, thick] (5) at (2.2, -2.2) {$3$};
            \node[state, minimum size=0.25cm, thick] (6) at (2.2, -3.3) {$3$};

            \node[state, minimum size=0.25cm, thick] (7) at (1.1, -1.1) {$5$};
            \node[state, minimum size=0.25cm, thick] (8) at (1.1, -2.2) {$6$};
            \node[state, minimum size=0.25cm, thick] (9) at (1.1, -3.3) {$7$};
            \node[state, minimum size=0.25cm, thick] (10) at (3.3, -3.3) {$6$};
            
            \path[thick] (1) edge (2);
            \path[thick] (2) edge (3);
            \path[thick] (3) edge (4);
            \path[thick] (4) edge (5);
            \path[thick] (5) edge (6);

            \path[thick] (4) edge (7);
            \path[thick] (5) edge (8);
            \path[thick] (6) edge (9);
            \path[thick] (6) edge (10);
        \end{tikzpicture}
    }
    \hspace{0.4cm}
    \subfloat[] {
        \begin{tikzpicture}
            \node[state, minimum size=0.25cm, thick] (1) at (0, -2.2) {$8$};
            \node[state, minimum size=0.25cm, thick] (2) at (0, -1.1) {$2$};
            \node[state, minimum size=0.25cm, thick] (3) at (1.1, -1.1) {$5$};
            \node[state, minimum size=0.25cm, thick] (4) at (1.1, -2.2) {$3$};
            \node[state, minimum size=0.25cm, thick] (5) at (2.2, -2.2) {$3$};
            \node[state, minimum size=0.25cm, thick] (6) at (1.1, -3.3) {$3$};

            \node[state, minimum size=0.25cm, thick] (7) at (3.3, -2.2) {$7$};
            \node[state, minimum size=0.25cm, thick] (8) at (2.2, -3.3) {$6$};
            \node[state, minimum size=0.25cm, thick] (9) at (0, -3.3) {$6$};
            \node[state, minimum size=0.25cm, thick] (10) at (1.1, -4.4) {$5$};
            
            \path[thick] (1) edge (2);
            \path[thick] (2) edge (3);
            \path[thick] (3) edge (4);
            \path[thick] (4) edge (5);
            \path[thick] (4) edge (6);

            \path[thick] (5) edge (7);
            \path[thick] (5) edge (8);
            \path[thick] (6) edge (9);
            \path[thick] (6) edge (10);
        \end{tikzpicture}
        \label{algo_2_example_h}
    }
    \hspace{0.4cm}
    \subfloat[] {
        \begin{tikzpicture}
            \node[state, minimum size=0.25cm, thick] (1) at (0, -2.2) {$8$};
            \node[state, minimum size=0.25cm, thick] (2) at (0, -1.1) {$2$};
            \node[state, minimum size=0.25cm, thick] (3) at (1.1, -1.1) {$5$};
            \node[state, minimum size=0.25cm, thick] (4) at (1.1, -2.2) {$3$};
            \node[state, minimum size=0.25cm, thick] (5) at (2.2, -2.2) {$3$};
            \node[state, minimum size=0.25cm, thick] (6) at (1.1, -3.3) {$3$};

            \node[state, minimum size=0.25cm, thick] (7) at (3.3, -2.2)  {$7$};
            \node[state, minimum size=0.25cm, thick] (8) at (2.2, -3.3) {$5$};
            \node[state, minimum size=0.25cm, thick] (9) at (0, -3.3) {$6$};
            \node[state, minimum size=0.25cm, thick] (10) at (1.1, -4.4) {$6$};
            
            \path[thick] (1) edge (2);
            \path[thick] (2) edge (3);
            \path[thick] (3) edge (4);
            \path[thick] (4) edge (5);
            \path[thick] (4) edge (6);

            \path[thick] (5) edge (7);
            \path[thick] (5) edge (8);
            \path[thick] (6) edge (9);
            \path[thick] (6) edge (10);
        \end{tikzpicture}
    }
    \caption{All nine isomorphism classes corresponding to the trees constructible by Algorithm 2 for $D_2 = (8, 7, 6, 6, 5, 5, 3, 3, 3, 2, 1, 1, \ldots, 1)$. Each vertex is labelled by its degree and all the leaves are left out for the sake of brevity.}
    \label{algo_2_example}
\end{figure}

\section*{Acknowledgements}

The authors would like to express their gratitude to Imre Leader for all of his supporting comments and remarks.

\section*{Conflict of interest}

The authors declare that they have no conflict of interest.

\end{document}